\documentclass[12pt]{amsart}

\usepackage{amsmath, amsthm, amssymb}
\input xy
\xyoption{all}
\usepackage{mathrsfs}
\usepackage{enumerate}

\newtheorem{theorem}[subsubsection]{Theorem}

\theoremstyle{definition}

\newtheorem{remark}[subsubsection]{Remark}

\newcommand{\Z}{\ensuremath{\mathbb{Z}}}

\newcommand{\C}{\ensuremath{\mathbb{C}}}

\renewcommand{\P}{\ensuremath{\mathbb{P}}}

\newcommand{\M}{\ensuremath{\overline{\mathcal{M}}}}
\renewcommand{\O}{\ensuremath{\mathcal{O}}}
\newcommand{\X}{\ensuremath{\mathcal{X}}}
\newcommand{\ev}{\ensuremath{\textrm{ev}}}
\newcommand{\vir}{\ensuremath{\textrm{vir}}}

\newcommand{\ch}{\ensuremath{\textrm{ch}}}

\newcommand{\CZr}{\ensuremath{[\mathbb{C}/\mathbb{Z}_r]}}
\renewcommand{\a}{\ensuremath{\mathbf{a}}}

\begin{document}

\title{Relations on $\M_{g,n}$ via orbifold stable maps}
\author{Emily Clader}
\date{}
\thanks{Partially supported by FRG grant DMS-1159265 and RTG grant DMS-1045119.}

\begin{abstract}
Using the equivariant virtual cycle of the moduli space of stable maps to $[\C/\Z_r]$, or equivalently, the vanishing of high-degree Chern classes of a certain vector bundle over the moduli space of stable maps to $B\Z_r$, we derive relations in the Chow ring of $\M_{g,n}(B\Z_r,0)$.  These push forward to yield tautological relations on $\M_{g,n}$. 
\end{abstract}

\maketitle

\section{Introduction}

The tautological ring of the moduli space of curves is the minimal family of subrings $R^*(\M_{g,n})$ inside the cohomology (or Chow ring) of each $\M_{g,n}$ that is stable under pullbacks and pushforwards via the forgetful and attaching maps between these moduli spaces.  In particular, the kappa, psi, lambda, and boundary classes, which play important roles in Gromov-Witten theory, are all elements of the tautological ring.

Relations between tautological classes in $\mathcal{M}_g$ were initially studied by Mumford \cite{Mumford} in the 1980s.  More recently, ideas from topological string theory have been applied to develop a more complete picture.  In 2000, Faber and Zagier conjectured the so-called FZ relations on $\mathcal{M}_g$, which were proved via the geometry of stable quotients in 2010 \cite{PP}.  A remarkable generalization occurred in 2013, when Pandharipande-Pixton-Zvonkine \cite{PPZ} used Witten's $3$-spin theory and the method of quantization developed by Givental and Teleman \cite{GiventalSemisimple, Teleman} to prove that a set of relations previously conjectured by Pixton \cite{Pixton} holds in the cohomology of $\M_{g,n}$.  When restricted to $\mathcal{M}_g$, these relations recover the FZ relations, and indeed, all presently known tautological relations follow from those proved in \cite{PPZ}.

The $3$-spin relations \cite{PPZ} were cast in a new framework by Janda in \cite{Janda}, which outlines a general method for deriving tautological relations via the existence of certain limits when a semisimple Cohomological Field Theory moves toward a nonsemisimple point.  This method can be used to prove that the $3$-spin relations hold in the Chow ring of $\M_{g,n}$, and that they are equivalent to relations that arise out of the geometry of equivariant stable maps to $\P^1$.

The main result of the present paper is the following:

\begin{theorem}
\label{maintheorem}
Suppose that $a_1, \ldots, a_n \in \{0, 1, \ldots, r-1\}$ satisfy $0 \neq \sum_{i=1}^n a_i \equiv 0 \mod r$.  Then, for any $d$ such that
\[\frac{1}{r}\sum_{i=1}^n a_i +g -1 < d \leq 3g-3+n,\]
one has the following relations in $A^d(\M_{g,\a}(B\Z_r, 0))$:

\[0 = \hspace{-0.4cm} \sum_{\substack{d_1 + \cdots + d_m = d\\ d_i, m \geq 1}} \frac{1}{m! \prod_{i} d_i(d_i+1)} \prod_{i=1}^m \left(B_{d_i+1}(0)\kappa_{d_i} - \sum_{j=1}^n B_{d_i+1}\left(\frac{a_j}{r}\right)\psi_j^{d_i} +\right.\]
\[\left.\frac{r}{2} \sum_{\substack{0 \leq l \leq g\\ I \subset [n]}}B_{d_i+1}\left(\frac{q_{l,I}}{r}\right) p^*i_{(l,I)*}(\gamma_{d_i-1}) + \frac{r}{2} \sum_{q=0}^{r-1} B_{d_i+1}\left(\frac{q}{r}\right)j_{(irr,q)*}(\gamma_{d_i-1})\right).\]

\vspace{0.2cm}

\noindent If $a_1 = \cdots = a_n = 0$, the same equations hold for $g > 0$ on components of $\M_{g,\mathbf{0}}(B\Z_r, 0)$ corresponding to curves with a {\it nontrivial} $r$-torsion line bundle.

In particular, the pushfowards of these equations via the map
\[\rho: \M_{g,\a}(B\Z_r, 0) \rightarrow \M_{g,n}\]
give tautological relations $A^d(\M_{g,n})$.
\end{theorem}

The definitions of the classes appearing in the theorem are reviewed in Section \ref{ChiodosFormula}.  We comment that, while the pushed-forward relations should follow from Pixton's 3-spin relations (which, conjecturally, generate all relations in $R^*(\M_{g,n})$), the initial equations in $\M_{g,\a}(B\Z_r, 0)$ are not pulled back from the moduli space of curves; see Remark \ref{pullback}.

These relations can be viewed, analogously to the context of \cite{Janda}, as a consequence of the existence of the nonequivariant limit in the Cohomological Field Theory associated to the equivariant Gromov-Witten theory of the orbifold $[\C/\Z_r]$.  Alternatively, they can be derived easily from the vanishing of high-degree Chern classes of a certain vector bundle on $\M_{g,n}(B\Z_r, 0)$, which are expressible in terms of tautological classes by Grothendieck-Riemann-Roch and the work of Chiodo \cite{ChiodoTowards}.  While the second perspective is more direct and leads naturally to a family of generalizations, we suspect that the equivariant point-of-view is important in understanding the connection between our results and the 3-spin relations.

\subsection{Future work}

In \cite{CJ}, we will study an important application of Theorem \ref{maintheorem}: it can be used to prove a conjecture of Pixton, namely that his proposed formula for the double ramification cycle vanishes in cohomological degrees past $g$ (see Section 4 of \cite{Cavalieri} for more details on Pixton's formula).  We hope, furthermore, to explicitly exhibit how the relations on $\M_{g,n}$ obtained from Theorem \ref{maintheorem} can be expressed in terms of the 3-spin relations.

\subsection{Outline of the paper}

In Section \ref{preliminaries}, we review the basic definitions from orbifold Gromov-Witten theory that we will require.  Most importantly, we express the equivariant virtual cycle for (most) components of the moduli space of stable maps to $[\C/\Z_r]$ as the Chern class of a vector bundle over $\M_{g,n}(B\Z_r,0)$.  Via Chiodo's Grothendieck-Riemann-Roch formula \cite{ChiodoTowards}, this affords us an expression for the equivariant virtual cycle in terms of tautological classes, which we present in Section \ref{relations} and use to prove Theorem \ref{maintheorem}.  We conclude with a discussion of a different perspective on the theorem, which circumvents the equivariant theory and the orbifold $\CZr$ entirely, and which can be naturally generalized to other moduli spaces $\M_{g,n}(BG, 0)$.

\subsection{Acknowledgments}

The author is particularly indebted to Y. Ruan and F. Janda for numerous invaluable conversations and insights.  The author would also like to thank A. Chiodo, R. Pandharipande, D. Petersen, A. Pixton, D. Ross, and D. Zvonkine for useful conversations and comments.  A computer program by D. Johnson for calculating top intersections on $\M_{g,n}$ was very helpful for studying cases of the main theorem. 

\section{Preliminaries on $[\C/\Z_r]$ and its Gromov-Witten theory}
\label{preliminaries}

\subsection{The orbifold and its cohomology}

By $[\C/\Z_r]$, we mean the stack quotient of $\C$ by the action of $\Z_r$ via multiplication by $r$th roots of unity.

The orbifold (or Chen-Ruan) cohomology of $\CZr$ is, as a vector space, the cohomology of the inertia stack $I \CZr$, whose objects are pairs $(x,g)$ with $x \in \CZr$ and $g$ an element of the isotropy group at $x$.  Explicitly,
\[I\CZr = \CZr \sqcup \bigsqcup_{i=1}^{r-1} B\Z_r,\]
in which $B\Z_r$ denotes the stack consisting of a single point with $\Z_r$ isotropy.  Thus,
\[H^*_{CR}(\CZr;\C) = \C\{\zeta_0, \zeta_1, \ldots, \zeta_{r-1}\},\]
where $\zeta_0$ is the constant function $1$ on $\CZr$ and $\zeta_i$ is the constant function $1$ on the $i$th copy of $B\Z_r$.  The component of the inertia stack whose cohomology is generated by $\zeta_0$ is called the {\it nontwisted sector}, while the other components are referred to as {\it twisted sectors}.

\subsection{Stable maps to orbifolds}

For an orbifold $\X$, let $\M_{g,n}(\X,d)$ denote the moduli stack of $n$-pointed genus-$g$ orbifold stable maps to $\X$ of degree $d$.  This stack parameterizes families of genus-$g$, $n$-pointed orbifold curves $C$ equipped with a representable morphism $f: C \rightarrow \X$ for which the corresponding map between coarse underlying spaces is a stable map of degree $d$.  The condition of representability means that the induced homomorphism on isotropy groups at every point is injective.

The precise definition of an orbifold curve can be found, for example, in \cite{AV}.  For the present, we recall just two particularly important features.  First, orbifold curves are allowed isotropy only at marked points and nodes.  Second, the orbifold structure at nodes is required to be {\it balanced}, which means that \'etale locally near each node, $C$ has the form
\[\left[\{xy = 0\}/\Z_{\ell}\right]\]
with $\Z_{\ell}$ acting by $\xi(x,y) = (\xi x, \xi^{-1} y)$.

The moduli stack $\M_{g,n}(\X,d)$ admits a perfect obstruction theory relative to the Artin stack of prestable pointed orbicurves, and this obstruction theory can be used to define a virtual fundamental class
\[[\M_{g,n}(\X,d)]^{\vir} \in H^*(\M_{g,n}(\X,d); \C).\]
As in the non-orbifold case, the perfect obstruction theory is given by the object $(R^{\bullet}\pi_* f^* T\X)^{\vee}$ in the derived category, where
\[\xymatrix{
\mathcal{C} \ar[r]^{f}\ar[d]_{\pi} & \X\\
\M_{g,n}(\X,d)
}\]
is the universal family over the moduli space.

Also analogously to the non-orbifold theory, there are evaluation morphisms, but in this case they map to the inertia stack:
\[\ev_i: \M_{g,n}(\X,d) \rightarrow I\X.\]
Specifically, the image of $(f: C \rightarrow \X)$ can be viewed as the point $(x,g) \in I\X$, where $x$ is the image of the $i$th marked point $p_i$ and $g$ is the image of the canonical generator of the isotropy group at $p_i$ under the homomorphism of isotropy groups $G_{p_i} \rightarrow G_x$ induced by $f$.

There is a decomposition of $\M_{g,n}(\X,d)$ according to the images of the evaluation maps:
\begin{equation}
\label{decomposition}
\M_{g,n}(\X,d) = \bigsqcup_{a_1, \ldots, a_n \in \mathcal{I}} \ev_1^{-1}(\X_{a_1}) \cap \cdots \cap \ev_n^{-1}(\X_{a_n}),
\end{equation}
where
\[I\X = \bigsqcup_{i \in \mathcal{I}} \X_{i}\]
decomposes $I\X$ into twisted sectors.  We denote
\[\M_{g,\mathbf{a}} (\X,d) = \ev_1^{-1}(\X_{a_1}) \cap \cdots \cap \ev_n^{-1}(\X_{a_n}),\]
where $\mathbf{a} = (a_1, \ldots, a_n)$.  Each of these components obtains a virtual fundamental class by restriction.

\subsection{Stable maps to $\CZr$}

In the special case where $\X = \CZr$, there are no stable maps of positive degree.  The decomposition (\ref{decomposition}) of the moduli space $\M_{g,n}(\CZr, 0)$ is written
\[\M_{g,n}(\CZr, 0) = \bigsqcup_{a_1, \ldots, a_n \in \Z_r} \M_{g,\mathbf{a}} (\CZr, 0),\]
where $\M_{g,\mathbf{a}}(\CZr,0)$ is the component in which the $i$th marked point maps to the twisted sector indexed by $\zeta_i$.  (We identify $\Z_r$ with the set $\{0,1, \ldots, r-1\}$ whenever it appears as an index set.)

We observe that the substack $\M_{g, \mathbf{a}}(\CZr, 0)$ is:
\begin{enumerate}
\item Compact if and only if not every $a_i$ is equal to zero.\footnote{There will, however, be at least one compact component of $\M_{g,(0,\ldots, 0)}(\CZr,0)$ when $g >0$; see Remark \ref{zero}.}
\item Nonempty if and only if $\sum_{i=1}^n a_i \equiv 0 \mod r$.
\end{enumerate}
The proof of the first of these assertions is a straightforward consequence of the compactness of $\M_{g,n}(\X, 0)$ for compact targets, while the second is a standard fact in orbifold Gromov-Witten theory; alternatively, it will follow from the description of the moduli space in terms of line bundles given below.

Unless otherwise specified, we will always work only on components for which not every $a_i$ is zero, to ensure compactness.  In this case, $\M_{g, \a}(\CZr, 0)$ is isomorphic to $\M_{g,\a}(B\Z_r, 0)$.  However, their obstruction theories are different; while there are obstructions to deformations of stable maps to $\CZr$ parameterized by $H^1(C, f^*N_{B\Z_r/\CZr})$, the moduli space of stable maps to $B\Z_r$ is smooth and unobstructed.  This leads to the following relationship between virtual cycles:
\begin{equation}
\label{twisted1}
[\M_{g,\a}(\CZr,0)]^{\vir} = [\M_{g,\a}(B\Z_r, 0)] \frown e(R^1\pi_*f^*N_{B\Z_r/\CZr}).
\end{equation}

Another perspective on stable maps to $B\Z_r$ can be used to make equation (\ref{twisted1}) more explicit.  The data of an orbifold stable map $C \rightarrow B\Z_r$ is equivalent to an orbifold line bundle $L$ on $C$ satisfying $L^{\otimes r} \cong \O_C$.  The map $f: C \rightarrow B\Z_r$ corresponds to the line bundle $f^*N$, where $N$ is the topologically trivial line bundle over $B\Z_r$ on which $\Z_r$ acts by multiplication by $r$th roots of unity.

From this point of view, $f^*N_{B\Z_r/\CZr}$ corresponds to the universal line bundle $\mathscr{L}$ on the universal curve over $\M_{g,n}(B\Z_r, 0)$.  Furthermore, the monodromy $a_i$ becomes the {\it multiplicity} of $L$ at the $i$th marked point--- that is, $a_i$ is the weight of the action of the isotropy group at the marked point $p_i$ on the fiber of $L$.

Much more about orbifold line bundles and their tensor powers can be found, for example, in \cite{Chiodo}.  One crucial fact is the relationship between an orbifold line bundle $L$ and its pushforward $|L|$ under the map $\epsilon: C \rightarrow |C|$ to the coarse underlying curve.  Suppose that $C_0$ is an irreducible component of $C$ with special points $x_1, \ldots, x_k$ at which $L$ has multiplicities $a_1, \ldots, a_k$.  Then, if $L^{\otimes r} \cong \epsilon^*M$ for a line bundle $M$ on $|C|$, the coarse underlying line bundle $|L| = \epsilon_*L$ satisfies
\[|L|^{\otimes r} \cong M \otimes \left( -\sum_{i=1}^k a_i [x_i]\right)\]
when restricted to $|C_0|$.

In the case of interest, $M = \O_C$, and the above formula implies that if $C$ has no nodes, then
\begin{equation}
\label{degree}
\deg(|L|) = -\frac{1}{r}\sum_{i=1}^n a_i,
\end{equation}
and in general, an analogous formula holds on each component of the normalization of $C$.  There are two important consequences of (\ref{degree}).  First, since $|L|$ is a line bundle on a non-orbifold curve, it must have integral degree, so we recover the fact mentioned previously that $\sum_{i=1}^n a_i \equiv 0 \mod r$.  Second, because at least one $a_i$ is nonzero, there is at least one component on which $\deg(|L|) < 0$.  Since $\deg(|L|) \leq 0$ on every component, this implies that $H^0(L) = H^0(|L|) = 0$.

Combining these observations, (\ref{twisted1}) gives the following expression for the virtual cycle of the moduli space of stable maps to $\CZr$ in terms of twisted theory over $B\Z_r$:
\begin{equation}
\label{twisted2}
[\M_{g,\a}(\CZr,0)]^{\vir} = [\M_{g,\a}(B\Z_r, 0)] \frown e(-R^{\bullet}\pi_*\mathscr{L}).
\end{equation}
This expression, and its equivariant version (\ref{equivariant}), will be central to the computations that follow.

\begin{remark}
\label{zero}
If $a_1 = \cdots = a_n = 0$, then one has $\text{deg}(L) = 0$ on every component of $C$, so the above argument that $H^0(L) = 0$ fails and (\ref{twisted2}) is no longer valid in general.  However, in this case, the moduli space $\M_{g,\mathbf{0}}(B\Z_r, 0)$ of $r$-torsion line bundles splits into components, on each of which $L$ is either always trivial or always nontrivial.  On nontrivial components, which exist as long as $g>0$, an $r$-torsion line bundle can have no global sections.   Thus, (\ref{twisted2}) still holds on these components of the moduli space.
\end{remark}

\subsection{Equivariant theory}

There is an action of $\C^*$ on $\CZr$ via multiplication, and this induces an action on $\M_{g,n}(\CZr,0)$.  The perfect obstruction theory is $\C^*$-equivariant, and as a result, one obtains an equivariant version of the virtual fundamental class.

In terms of the expression (\ref{twisted2}) on components where not every $a_i$ is zero, the equivariant virtual fundamental class is
\begin{equation}
\label{equivariant}
\M_{g,\a}(\CZr,0)]^{\C^*, \vir} = [\M_{g,\a}(B\Z_r, 0)]^{\C^*} \frown e^{\C^*}(-R^{\bullet}\pi_*\mathscr{L}),
\end{equation}
in which $\C^*$ acts trivially on $\M_{g,\a}(B\Z_r, 0)$ and by multiplication on the fibers of the vector bundle $-R^{\bullet}\pi_*\mathscr{L}$.  Despite the triviality of the action on these components, the above expression is useful, because the equivariant Euler class can easily be written in terms of Chern characters, which affords the possibility of applying the Grothendieck-Riemann-Roch formula.

\section{Proof of Theorem \ref{maintheorem}}
\label{relations}

In this section, we write (\ref{equivariant}) explicitly in terms of tautological classes on $\M_{g, \a}(B\Z_r, 0)$, and we derive Theorem \ref{maintheorem} via the existence of the nonequivariant limit.

\subsection{Introducing Chern characters}

Any multiplicative invertible characteristic class can be written in terms of Chern characters.  For the $\C^*$-equivariant Euler class, one obtains
\begin{equation}
\label{euler}
e^{\C^*}((-V)^{\vee}) = \exp\left(\sum_{d=0}^{\infty} s_d \ch_d([V])\right)
\end{equation}
for
\[s_d = \begin{cases} -\ln(\lambda) & d=0\\ \;\\ \displaystyle\frac{(d-1)!}{\lambda^d} & d>0,\end{cases}\]
where $\lambda$ is the equivariant parameter.

The rank of the vector bundle $-R^{\bullet}\pi_*\mathscr{L}$ can be computed using the Riemann-Roch formula; it equals
\[\text{rank}(-R^{\bullet}\pi_*\mathscr{L}) = h^1(|L|) = -\deg(|L|) + g -1 = \frac{1}{r}\sum_{i=1}^n a_i + g -1\]
for any element $(C,L)$ of $\M_{g,\a}(B\Z_r, 0)$.  Combining this with (\ref{equivariant}) and (\ref{euler}) expresses the equivariant virtual cycle of $\M_{g, \mathbf{a}}(\CZr, 0)$ as the Poincar\'e dual of
\begin{equation}
\label{twisted3}
e^{\C^*}(-R^{\bullet}\pi_*\mathscr{L}) = (-1)^{\text{deg}}\lambda^{\frac{1}{r}\sum a_i + g-1} \exp\left(\sum_{d=1}^{\infty} \frac{(d-1)!}{\lambda^d} \ch_d(R^{\bullet}\pi_*\mathscr{L})\right).
\end{equation}
Here, $\text{deg}$ denotes the nonequivariant degree (that is, the cohomological degree of an equivariant class after setting $\lambda =1$), so $(-1)^{\text{deg}}$ should be understood as an operator on $A^*_{\C^*}(\M_{g, \mathbf{a}}(\CZr, 0))$.

\subsection{GRR for the universal $r$th root}
\label{ChiodosFormula}

By applying the orbifold Grothendieck-Riemann-Roch formula (see Appendix A of \cite{Tseng})--- or the ordinary GRR formula to coarse underlying curves--- the Chern characters of $R^{\bullet}\pi_*\mathscr{L}$ can be expressed in terms of tautological classes.

In \cite{ChiodoTowards}, Chiodo carries out this computation explicitly and substantially simplifies the result.  Strictly speaking, his computation is carried out on a slightly different moduli space from $\M_{g,\mathbf{a}}(B\Z_r, 0)$, parameterizing $r$th roots of $\O(-\sum_i a_i[x_i])$ on curves with no orbifold structure at marked points and balanced $\Z_r$ orbifold structure at every node.  In particular, the representability assumption does not appear in his moduli space, so it will only agree with $\M_{g,\mathbf{a}}(B\Z_r, 0)$ when all of the multiplicities at nodes are coprime to $r$.   However, there is always a morphism between his moduli space $\M_{g,n}^r$ and the moduli space $\M_{g,\mathbf{a}}(B\Z_r, 0)$ appearing here, which is invertible on the interior and has explicitly-computable ramification indices on the boundary.  Thus, we can state our formula on $\M_{g,\mathbf{a}}(B\Z_r, 0)$, but all of the classes in the formula will be pulled back from $\M_{g,n}^r$.

Corollary 3.1.8 of \cite{ChiodoTowards} states, in the present situation, that
\[\ch_d(R^{\bullet}\pi_*\mathscr{L}) = \frac{B_{d+1}(0)}{(d+1)!}\kappa_d - \sum_{i=1}^n \frac{B_{d+1}(\frac{a_i}{r})}{(d+1)!}\psi_i^d \;+ \hspace{3.5cm}\]
\[\hspace{0.75cm}\frac{r}{2} \sum_{\substack{0 \leq l \leq g\\ I \subset [n]}} \frac{B_{d+1}\left(\frac{q_{l,I}}{r}\right)}{(d+1)!} p^*i_{(l,I)*}(\gamma_{d-1}) + \frac{r}{2} \sum_{q=0}^{r-1} \frac{B_{d+1}(\frac{q}{r})}{(d+1)!}j_{(irr,q)*}(\gamma_{d-1}).\]

Let us review the notation appearing in this formula.  First, $B_{d+1}(x)$ are the Bernoulli polynomials, defined by the generating function
\[\frac{te^{xt}}{e^t -1} = \sum_{n=0}^{\infty} B_n(x) \frac{t^n}{n!}.\]
The kappa and psi classes are defined by pullback under the map
\[p: \M_{g,n}(B\Z_r, 0) \rightarrow \M_{g,n}(r),\]
where $\M_{g,n}(r)$ is the moduli space of $r$-stable, $n$-pointed, genus-$g$ curves.  The notion of $r$-stability is defined precisely in Definition 2.1.1 of \cite{ChiodoTowards}, but essentially, it refers to orbifold curves in which the marked points have trivial isotropy and the nodes are balanced with $\Z_r$ isotropy.  On this moduli space, kappa and psi classes are defined as usual:
\[\kappa_d = \pi_*(c_1(\omega_{\log})^d), \hspace{1cm} \psi_i = c_1(\omega|_{p_i}),\]
where $\omega_{\log} = \omega_{C/X}(\sum_{i=1}^n [p_i])$.

Let $Z$ be the codimension-two locus inside the universal curve over $\M_{g,n}(r)$ consisting of nodes in singular fibers.  There is a decomposition
\begin{equation}
\label{Z}
Z = \bigsqcup_{\substack{0 \leq l \leq g\\ I \subset [n]}} Z_{(l, I)} \sqcup Z_{irr},
\end{equation}
where $Z_{(l, I)}$ consists of nodes separating the curve $C$ into a component of genus $l$ containing the marked points in $I$ and a component of genus $g-l$ containing the other marked points, while $Z_{irr}$ consists of nonseparating nodes.  Let $Z' \rightarrow Z$ be the two-fold cover given by nodes together with a choice of branch.  The decomposition (\ref{Z}) induces an analogous decomposition of $Z'$, and the morphisms
\[i_{(l,I)}: Z'_{(l, I)} \rightarrow \M_{g,n}(r)\]
are given by this two-fold cover, followed by the inclusion into the universal curve and the projection to $\M_{g,n}(r)$.

The index $q_{l,I}$ is the multiplicity at the chosen branch of the node for any point in $Z'_{(l,I)}$.  The equation (\ref{degree}) determines this multiplicity from $l$ and $I$; specifically, it is the unique number in $\{0,1, \ldots, r-1\}$ satisfying
\[q_{l,I} + \sum_{i \in I} a_i \equiv 0 \mod r.\]
If $\psi$ is the first Chern class of the line bundle over $Z'_{(l,I)}$ whose fiber is the cotangent line to the chosen branch of the node, and $\hat{\psi}$ is the first Chern class of the bundle whose fiber is the cotangent line to the opposite branch, then $\gamma_d$ is defined by
\[\gamma_d = \sum_{i+j=d} (-\psi)^i \hat{\psi}^j.\]

Finally, let $Z'_{(irr,q)}$ be the locus inside the universal curve over $\M_{g,n}^r$ consisting of nonseparating nodes with a choice of branch such that the multiplicity of the line bundle $L$ at the chosen branch is equal to $q \in \{0, 1,\ldots, r-1\}$.  It should be noted that $q$ is not determined by the topology of the underlying curve alone, as the multiplicity was in the case of separating nodes, so we must view $Z'_{(irr,q)}$ as lying over $\M_{g,n}^r$ rather than over $\M_{g,n}(r)$.  We have morphisms
\[j_{(irr,q)}: Z'_{(irr,q)} \rightarrow \M_{g,n}^r\]
given, as before, by the two-fold cover, inclusion into the universal curve, and projection.  The class $\gamma_d$ is defined by exactly the same formula as previously.  By pullback under the birational map discussed above, the class $j_{(irr,q)*}(\gamma_{d-1})$ is viewed as lying on $\M_{g,\mathbf{a}}(B\Z_r,0)$, and, by abuse of notation, it is still denoted $j_{(irr,q)*}(\gamma_{d-1})$.

\subsection{Tautological relations}

Equipped with this formula, the proof of Theorem \ref{maintheorem} is immediate from (\ref{twisted3}):

\begin{proof}[Proof of Theorem \ref{maintheorem}]
The existence of the nonequivariant limit in (\ref{twisted3}) implies that any term with a negative power of $\lambda$ must vanish, which proves the theorem whenever not every $a_i$ is zero.  The case $a_1 = \cdots = a_n = 0$ follows from the same argument after applying Remark \ref{zero}.
\end{proof}

\begin{remark}
\label{pullback}
While pushforward yields relations in $A^d(\M_{g,n})$, the original relations in $A^d(\M_{g,\mathbf{a}}(B\Z_r, 0))$ do not arise by pullback from the moduli space of curves.  The reason for this is the appearance of the classes $j_{(irr,q)*}(\gamma_{d-1})$, which are not pulled back because the multiplicity at a nonseparating node is not determined by the topology of the coarse underlying curve.
\end{remark}

\subsection{An alternative perspective}

There is a different way to understand the proof of Theorem \ref{maintheorem} that circumvents $\CZr$ entirely.  Namely, let
\[M = \M_{g,\mathbf{a}}(B\Z_r, 0),\]
and let $V$ be the vector bundle $-R^{\bullet}\pi_*\mathscr{L}$ on this stack.

The total Chern class of a vector bundle can always be written in terms of its Chern characters:
\[c(V) = \exp\left(\sum_{d \geq 1} (-1)^{d+1}(d-1)! \cdot \text{ch}_d(V)\right).\]
While the exponential appears to have contributions in every degree, the total Chern class clearly vanishes beyond the rank of the vector bundle.  Theorem \ref{maintheorem} is a consequence of this vanishing when expanded in terms of the Chern characters on the right-hand side.

In particular, only two facts are necessary for the proof:
\begin{enumerate}
\item One has a moduli space $M$ with a map to $\M_{g,n}$;
\item On $M$, there is a vector bundle $V$ whose Chern characters can be expressed in terms of classes whose pushforwards to $\M_{g,n}$ are tautological.
\end{enumerate}

Any such situation will yield an analogous collection of tautological relations.  For example, consider the moduli space $\M_{g,n}(BG, 0)$ for any finite group $G$.  Fixing a character
\[\rho: G \rightarrow \C^*\]
yields an orbifold line bundle $N_{\rho}$ on $BG$, which is topologically trivial and for which the isotropy group $G$ acts on the fiber via $\rho$.  If $\mathcal{L}_{\rho}$ denotes the pullback of $N_{\rho}$ to the universal curve over $\M_{g,n}(BG, 0)$, then, for the same reason as above, $R^1\pi_*\mathcal{L}_{\rho}$ will be a vector bundle on any component $\M_{g,\a}(BG, 0)$ on which at least one of the monodromies $a_i$ is not in the kernel of $\rho$.

Thus, one can take $M$ to be such a component of $\M_{g,n}(BG, 0)$ and $V$ to be the restriction of $R^1\pi_*\mathcal{L}_{\rho}$ to this component.  By an application of the orbifold Grothendieck-Riemann-Roch theorem precisely analogous to Chiodo's, an expression for the Chern characters of $R^1\pi_*\mathcal{L}_{\rho}$ in terms of tautological classes can be obtained.  In this way, any choice of $G$ and $\rho$ as above can be used to derive relations.

\bibliographystyle{abbrv}
\bibliography{TautologicalBibClader}
\nocite{*}

\end{document}